\documentclass[11pt]{article}

\usepackage{amsmath}
\usepackage{amsfonts}
\usepackage{color}
\usepackage{amssymb}
\usepackage{graphicx}
\usepackage{enumerate}
\usepackage[all]{xy}

 \allowdisplaybreaks

\begin{document}

\newtheorem{problem}{Problem}

\newtheorem{theorem}{Theorem}[section]
\newtheorem{corollary}[theorem]{Corollary}
\newtheorem{definition}[theorem]{Definition}
\newtheorem{conjecture}[theorem]{Conjecture}
\newtheorem{question}[theorem]{Question}
\newtheorem{lemma}[theorem]{Lemma}
\newtheorem{proposition}[theorem]{Proposition}
\newtheorem{quest}[theorem]{Question}
\newtheorem{example}[theorem]{Example}
\newenvironment{proof}{\noindent {\bf
Proof.}}{\rule{2mm}{2mm}\par\medskip}
\newenvironment{proofof}{\noindent {\bf
Proof of the Theorem 6.1.}}{\rule{2mm}{2mm}\par\medskip}
\newcommand{\remark}{\medskip\par\noindent {\bf Remark.~~}}
\newcommand{\pp}{{\it p.}}
\newcommand{\de}{\em}
\newcommand{\norm}[1]{\left\lVert#1\right\rVert}

\title{  {The spectral radius  of graphs with no intersecting triangles}\thanks{S. M. Cioab\u{a} is supported by NSF grants DMS-1600768 and CIF-1815922 and a JSPS Fellowship. L. Feng  is supported by  NSFC (Nos.11871479, 11671402),  Hunan Provincial Natural Science Foundation (2016JJ2138,  2018JJ2479) and  Mathematics and Interdisciplinary Sciences Project of CSU.   M. Tait is partially supported by NSF grant DMS-1855530. X.-D. Zhang is supported by  NSFC  (Nos.11531001, 11971311); the Montenegrin-Chinese Science and Technology Cooperation Project (No.3-12). \newline
 E-mail addresses: cioaba@udel.edu(S. Cioaba), fenglh@163.com (L. Feng),  michael.tait@villanova.edu (M. Tait) , $^\dag$xiaodong@sjtu.edu.cn (X.-D. Zhang).}}

\author{
 Sebastian Cioab\u{a}$^a$, Lihua Feng$^b$,  Michael Tait$^c$, Xiao-Dong Zhang$^{d\dag}$\\
{\small $^a$Department of Mathematical Sciences, University of Delaware}\\
{\small  Newark, DE 19716-2553, USA } \\
{\small $^b$School of Mathematics and Statistics, Central South University} \\
{\small New Campus,  Changsha, Hunan, 410083, P.R. China}\\
{\small $^c$Department of Mathematics and Statistics, Villanova University}\\
{\small Villanova, PA,  USA} \\
{\small $^d$School of Mathematical Sciences, MOE-LSC, SHL-MAC,
Shanghai Jiao Tong University}\\
{\small Shanghai 200240, P. R. China}}
\date{}
\maketitle

\vspace{-0.5cm}

\begin{abstract}
A graph on $2k+1$ vertices consisting of $k$ triangles which intersect in exactly one
common vertex is called a $k$-fan and denoted by $F_k$.
This paper aims to determine the graphs of order $n$  that have the maximum (adjacency) spectral radius among all graphs containing no $F_k$,  for $n$ sufficiently large.
 \end{abstract}



\section{Introduction}
In this paper, only simple and undirected graphs are considered. Let $G$ be a simple connected  graph with vertex set $V(G)=\{v_1, \ldots, v_n\}$ and edge set $E(G)=\{e_1, \ldots, e_m\}$.  Let $d(v_i)$ (or $d_G(v_i)$) be the degree of a vertex $v$ in $G$.
The \emph{adjacency matrix} of $G$ is $A(G)=(a_{ij})_{n \times n}$ with $a_{ij}=1$ if two vertices $v_i$ and $v_j$ are adjacent in $G$, and $a_{ij}=0$ otherwise.   The largest eigenvalue of $A(G)$, denoted by $\lambda (G)$ or $\lambda_1(G)$, is called the {\it spectral radius} of $G$.
In  spectral graph theory,   one of the  well-known problems  is the
Brualdi-Solheid problem \cite{brualdi}:  Given a set  of graphs, try to find a tight upper bound for the spectral radius in
this set  and characterize all  extremal graphs.
This problem is widely studied in the
literature for many classes of graphs, such as  graphs with cut vertices \cite{berman2001}, graphs with given diameter \cite{Hansen},  domination number
\cite{Stevanovic}, given size \cite{Stanley}, subgraphs of the hypercube \cite{cube}, graphs with Euler genus \cite{Ellingham} and given clique or independence number \cite{Wilf}.

The Brualdi--Solheid problem is with a Tur\'an-type flavor. The following problem regarding the adjacency  spectral radius
was proposed in \cite{NikiforovTuran}:
What is the maximum spectral radius of a graph $G$ on $n$
vertices without a subgraph isomorphic to a given graph $F$?
 For this problem,
 Fiedler and Nikiforov  \cite{FiedlerNikif} obtained tight sufficient conditions for
graphs to be Hamiltonian or traceable. Additionally, Nikiforov obtained spectral strengthenings of Tur\'an's theorem \cite{NikiforovTuran} and the K\H{o}vari-S\'os-Tur\'an theorem \cite{NikiforovKST} when the forbidden graphs are complete or complete bipartite respectively. This motivates further study for such question,
 see \cite{FengMonoshMath, FiedlerNikif,  WJLiuLAMA, NikiforovLAA10,   NikifSurvey}.

 The {\em Tur\'an number} of a graph $F$ is the maximum number of edges that may be in an $n$-vertex graph without a subgraph isomorphic to $F$, and is denoted by $\mathrm{ex}(n, F)$. A graph on $n$ vertices with no subgraph $F$ and with $\mathrm{ex}(n, F)$ edges is called an {\em extremal graph} for $F$ and we denote by $\mathrm{Ex}(n, F)$ the set of all extremal graphs on $n$ vertices for $F$. Understanding $\mathrm{ex}(n, F)$ and $\mathrm{Ex}(n, F)$ for various graphs $F$ is a cornerstone of extremal graph theory (see \cite{Alon, Conlon13, erdos1986, fox2011, FurediSurvey, Sidorenko} for surveys).

A graph on $2k+1$ vertices consisting of $k$ triangles which intersect in exactly one
common vertex is called a {\it $k$-fan} and denoted by $F_k$.

In \cite{Erdos95}, it is proved that
\begin{theorem}\label{Erdos95} \cite{Erdos95}
For every $k \geq 1$, and for every $n\geq 50k^2$, if a graph $G$ of order $n$ does not contain a copy of a $k$-fan, then
$e(G)\le \mathrm{ex}(n, F_k)$, where
\[\mathrm{ex}(n, F_k)= \left\lfloor \frac {n^2}{4}\right \rfloor+ \left\{
  \begin{array}{ll}
   k^2-k \quad~~  \mbox{if $k$ is odd,} \\
    k^2-\frac32 k \quad \mbox{if $k$ is even}.
  \end{array}
\right. \]
Furthermore, the number of edges is best possible.
\end{theorem}

The extremal graphs $G^i_{n,k}$ ($i=1,2$) of Theorem \ref{Erdos95} are as follows.
For odd $k$ (where $n\geq 4k-1$) $G^1_{n,k}$ is constructed by taking a complete bipartite graph with color classes of size $\lceil \frac{n}{2}\rceil$ and $\lfloor \frac{n}{2}\rfloor$ and embedding two vertex
disjoint copies of $K_k$ in one side.
For even $k$ (where now $n\geq 4k-3$) $G^2_{n,k}$ is constructed by taking a complete equi-bipartite graph and embedding
 a graph with $2k-1$ vertices, $k^2-\frac32 k$ edges with maximum degree $k-1$ in one side. The graphs $G^1_{n,k}$ is unique up to isomorphism but $G^2_{n,k}$ is not.

 \medskip

 Our goal is to give the spectral counterpart of Theorem \ref{Erdos95}.
 The main result of this paper is the following.

 \begin{theorem}\label{MainTH}
Let $G$ be a graph of order $n$ that does not contain a copy of a $k$-fan.  For sufficiently large $n$, if $G$ has the maximal spectral radius, then
$$G \in \mathrm{Ex}(n, F_k).$$
\end{theorem}

We note that using eigenvalue interlacing, one may form an equitable partition of a graph in $\mathrm{Ex}(n, F_k)$ and determine its spectral radius as the root of a degree $3$ (if $k$ is odd) or degree $4$ (if $k$ is even) polynomial. We at last point out that, during our proof, we use the triangle removal lemma, and so it is difficult to present exactly how large we need our $n$ to be.

\section{Some Lemmas}

Let $G$ be a simple graph with matching number $\beta(G)$ and maximum degree $\Delta(G)$. For given two integers $\beta$ and $\Delta$, define $f(\beta, \Delta)=\max\{|E(G)|: \beta(G)\leq \beta, \Delta(G)\leq \Delta \}$. Chv\'atal and Hanson \cite{Chvatal76} obtained the following result.

\begin{theorem}[Chv\'atal and Hanson \cite{Chvatal76}]\label{Chvatal76}
For every two positive integers $\beta \geq 1$ and $\Delta \geq 1$, we have
$$f(\beta, \Delta)= \Delta \beta +\left\lfloor\frac{\Delta}{2}\right\rfloor
 \left \lfloor \frac{\beta}{\lceil{\Delta}/{2}\rceil }\right \rfloor
 \leq \Delta \beta+\beta.$$
\end{theorem}
We will frequently use the following special case proved by Abbott, Hanson, and Sauer \cite{Abbott72}:
$$f(k-1,k-1) = \left\{
  \begin{array}{ll}
   k^2-k \quad~~  \mbox{if $k$ is odd,} \\
    k^2-\frac32 k \quad  \mbox{if $k$ is even}.
  \end{array}
\right.$$
The extremal graphs are exactly those we embedded into the Tur\'{a}n graph $T_{n,2}$  to obtain the extremal $F_k$-free graph $G^i_{n,k}$ ($i=1,2$).

Essential to our proof  are  the following two lemmas: the triangle removal lemma and a stability result of F\"uredi.

\begin{lemma}[Triangle Removal Lemma \cite{Erdos95,fox2011,Ruzsa76}]\label{triangleremoval2}
For each $\varepsilon>0$, there exists an  $N=N(\varepsilon)$ and $\delta  > 0$ such that every graph $G$  on $n$ vertices with $n\geq N$ with at
most $\delta n^3$  triangles  can be made triangle-free by removing at most $\varepsilon n^2$ edges.
\end{lemma}

\begin{lemma}[F\"uredi \cite{Furedi2015}]  \label{furedi20153}Suppose that $K_{3}\nsubseteq G$, $|V(G)|=n$, $s>0$ and
$e(G) = e(T_{n,2}) -s.$
Then there exists a bipartite subgraph $H$, $E(H)\subseteq E(G)$ such that
$e(H) \ge e(G) - s.$
\end{lemma}

The following lemma is needed in the sequel.

\begin{lemma}For any integer $n\geq 2$, we have
$$\frac{n}{2}-\sqrt{\left\lceil \frac{n}{2}  \right \rceil  \left\lfloor \frac{n}{2} \right \rfloor  }<\frac{1}{n}.$$
\end{lemma}

\section{The Proof of  Theorem~\ref{MainTH}}

Let $\mathcal{G}_{n,k}$ be the set of all $F_k$-free graphs of order $n$.
Let $G \in \mathcal{G}_{n,k}$ be a graph on $n$ vertices with maximum spectral radius. The aim of this section is to prove that $e(G) = \mathrm{ex}(n, F_k)$ for $n$ large enough.

 Let $\lambda_1$ be the spectral radius of $G$ and let $\mathbf{x}$ be a positive eigenvector for it. We  may normalize $\mathbf{x}$ so that it has maximum entry equal to $1$, and let $z$ be a vertex such that $\mathbf{x}_z = 1$. We prove the theorem iteratively, giving successively better lower bounds on both $e(G)$ and the eigenvector entries of all of the other vertices, until finally we can show that $e(G) = \mathrm{ex}(n, F_k)$.

 Let $H \in \mathrm{Ex}(n, F_k)$. Then since $G$ is the graph maximizing the spectral radius over all $F_k$-free graphs, in view of Theorem \ref{Erdos95},
 we must have
\begin{equation}\label{first lower bound1}
\lambda_1(G) \geq \lambda_1(H) \geq \frac{\mathbf{1}^T A(H) \mathbf{1}}{\mathbf{1}^T\mathbf{1}}
= 2\frac{  \left\lceil \frac{n}{2}  \right \rceil  \left\lfloor \frac{n}{2} \right \rfloor  + f(k-1, k-1)}{n} >\frac{n}{2}  .
\end{equation}

The proof of Theorem~\ref{MainTH}  is outlined as follows.

\begin{itemize}
\item We give a lower bound on $e(G)$ as a function of $\lambda_1$ and the number of triangles in $G$, which on first approximation gives a bound of roughly $\frac{n^2}{4} - O(kn)$.
\item Using the triangle removal lemma and F\"uredi's stability result, we show that $G$ has a very large maximum cut.
\item We show that no vertex has many neighbors on its side of the partition, and then we refine this by considering eigenvector entries to show that in fact no vertex has more than a constant number of neighbors on its side of the partition.
\item We show that no vertices have degree much smaller than $\frac{n}{2}$, and this allows us to refine our lower bound on both $e(G)$ and on the eigenvector entry of each vertex.
\item Once we know that all vertices have eigenvector entry very close to $1$, we may show that the partition is balanced. This shows that $G$ can be converted to a graph in $\mathrm{Ex}(n, F_k)$ by adding or removing a constant number of edges, and this allows us to show that $e(G) = \mathrm{ex}(n, F_k)$.
\end{itemize}

We now proceed with the details.  First  we prove a lemma which gives a lower bound on $e(G)$ in terms of $\lambda_1$ and the number of triangles in $G$.

\begin{lemma}\label{lower bound with triangles2}
If  $G$ has $t$ triangles, then
\[
e(G) \geq \lambda_1^2 - \frac{3t}{\lambda_1}.
\]
\end{lemma}
\begin{proof}
Let $\lambda_1$ be the spectral radius of $G$ and let $\mathbf{x}$ be a positive eigenvector scaled such that it has maximum entry equal to $1$, and let $z$ be a vertex with maximum eigenvector entry i.e., $\mathbf{x}_z = 1$.
Then
$\lambda_1x_u=\sum_{v\sim u}x_v$ and
$\lambda_1^2x_u=\sum_{v\sim u}\lambda_1x_v=\sum_{v\sim u}\sum_{w\sim v}x_w.$
We consider the following triple sum:
\[
\sum_{u\in V}\lambda_1^2x_u
 =\sum_{u\in V} \sum_{v\sim u}\sum_{w\sim v}x_w.
 \]
 The sum counts over all ordered walks on three vertices (with possible repetition), and adds the eigenvector entry of the last vertex. Instead of summing over ordered triples of vertices, we count by considering the first edge in the walk. If a given walk has first edge $uv$, then $\mathbf{x}_w$ will be counted by this edge exactly once if $w$ is adjacent to exactly one of $u$ or $v$ and exactly twice if $\{u,v,w\}$ forms a triangle. Therefore, the sum is equal to
 \begin{eqnarray*}
\sum_{u\sim v}\left( 2\sum_{\substack{w\sim u\\ w\sim v}} \mathbf{x}_w + \sum_{\substack{w\sim u\\ w\not\sim v}} \mathbf{x}_w + \sum_{\substack{w\sim v\\w\not\sim u}} \mathbf{x}_w\right)&=&\sum_{u\sim v}\left (\sum_{\substack{w\sim u\\ w\sim v}}x_w+\sum_{\substack{w\sim u\\ w\not\sim v}}x_w+\sum_{w\sim v}x_w \right )\\
&\le& \sum_{u\sim v} \left(\sum_{\substack{w\sim u\\ w\sim v}}x_w+\sum_{ w\not\sim v}x_w+\sum_{w\sim v}x_w  \right)\\
&\le &  \sum_{u\sim v} \left(\sum_{\substack{w\sim u\\ w\sim v}}x_w+\sum_{w\in V}x_w \right)\\
&\le &  \sum_{u\sim v}  \left(\sum_{\substack{w\sim u\\ w\sim v}}1+\sum_{w\in V}x_w \right)\\
&=& 3t+e(G)\sum_{w\in V}x_w.
\end{eqnarray*}
Hence
$$e(G)\ge \lambda_1^2-\frac{3t}{\sum_{w\in V}x_w}.$$
On the other hand,
$$\lambda_1=\lambda_1x_z=\sum_{w\sim z}x_w\le \sum_{w\in V}x_w.$$
Therefore
$$e(G)\ge \lambda_1^2-\frac{3t}{\sum_{w\in V}x_w}\ge \lambda_1^2-\frac{3t}{\lambda_1}.$$
So the assertion holds.\end{proof}

\begin{corollary}\label{edgelower2}
If  the  number of   triangles of $G$ is $t$, then
\[
e(G) \geq \lambda_1^2 - \frac{6t}{n }.
\]
\end{corollary}
\begin{proof}
In view of inequality  (\ref{first lower bound1}),  and the function $f(x)=  \lambda_1^2-\frac{3t}{x}$ is strictly increasing with respect to $x$, the assertion follows.
\end{proof}

\begin{lemma}\label{matchedge}
Suppose the matching number of a graph   $H$ of order $n$  is at most $k-1$. Then $e(H)\le kn$, i.e., $\mathrm{ex}(n, M_k)\le kn$, where $M_k$ is  a matching of size $k$.
\end{lemma}
\begin{proof}
By Theorem \ref{Chvatal76}, $e(H)\le f(k-1, n-1)\le (k-1)(n-1+1)< kn$.
\end{proof}

\begin{lemma}\label{maxcut}
Let $\varepsilon$ and $\delta$ be fixed positive constants with $\delta<\frac{1}{4k}$, $\varepsilon<\frac{\delta^2}{3}$. There exists an $N(\varepsilon, \delta, k)$ such that $G$ has a partition $V=S\cup T$ which gives a  maximum cut,  and
$$e(S, T)\ge \left  (\frac{1}{4}-\varepsilon \right)n^2$$
for $n\ge N(\varepsilon, \delta, k)$. Furthermore
\[
\left(\frac{1}{2} - \sqrt{\varepsilon} \right) n \leq |S|, |T| \leq \left(\frac{1}{2} + \sqrt{\varepsilon}\right) n.
\]
\end{lemma}
\begin{proof}
Since $G$ is $F_k$-free, the neighborhood of any vertex does not have $M_k$ (a matching of size $k$) as a subgraph. Thus
by Lemma~\ref{matchedge}, we can obtain the following upper bound for the number of triangles,
\[3t = \sum_{v\in V(G)} e(G[N(v)]) \leq \sum_{v\in V(G)} \mathrm{ex}(d(v)+1, M_k) \leq \sum_{v\in V(G)} \mathrm{ex}(n, M_k) \leq \sum_{v\in V(G)} kn=kn^2.
\]
This gives $t \leq \frac{kn^2}{3} $.
So  $t \leq  \frac{k}{3n}n^3\le \delta n^3$ for $n\ge  N_2\ge \frac{k}{3\delta}$.
From Lemma \ref{lower bound with triangles2}, we obtain
  \begin{equation}\label{maxsize}
e(G) \geq \frac{n^2}{4}-2kn.
  \end{equation}
    By Lemma \ref{triangleremoval2}, there exists an $N_1(\varepsilon,  k)$ such that  the graph $G_1$ obtained from $G$ by deleting
       at most $\frac{1}{10}\varepsilon n^2$ edges is $K_3$-free. For  $N=\max\{N_1, N_2\}$, the size of the   graph $G_1$ of order $n\ge N$ satisfies
    $e(G_1)\ge e(G)-\frac{1}{10}\varepsilon n^2$.
      Note that  $e(G_1)\le   e(T_{n,2})$ by T{u}r\'an's Theorem.  Define
     $$s\triangleq e(T_{n,2})-e(G_1)\ge 0.$$
    By Lemma \ref{furedi20153}, $G_1$ contains a bipartite subgraph $G_2$ such that $e(G_2)\ge e(G_1)-s$.
    Hence,    for $n$ sufficiently large, we have
    \begin{eqnarray*}
    e(G_2) &\ge & e(G_1)-s\\
       &=& 2e(G_1)-e(T_{n,2})\\
    &\ge & 2e(G)-e(T_{n,2})-\frac{1}{4}\varepsilon n^2\\
     &\ge& 2 \left(\frac{n^2}{4}-2kn \right)- \frac{n^2}{4}-\frac{1}{5}\varepsilon n^2\\
    &\ge &\left (\frac{1}{4}-\varepsilon \right) n^2.
    \end{eqnarray*}
        Therefore,   $G$ has  a partition  $V=S\cup T$ which gives a  maximum cut  such that
     \begin{equation}\label{maxcut1}
     e(S,T)\ge e(G_2)\ge \left  (\frac{1}{4}-\varepsilon \right) n^2.
    \end{equation}
    Furthermore, without loss of generality, we may assume that  $|S|\le |T|$. If $|S|<(\frac{1}{2}-\sqrt{\varepsilon})n$, then $|T|=n-|S|>
    (\frac{1}{2}+\sqrt{\varepsilon})n$. So
    $$e(S,T)\le |S||T| <  \left  (\frac{1}{2}-\sqrt{\varepsilon} \right)n  \left (\frac{1}{2}+\sqrt{\varepsilon} \right)n=\left (\frac{1}{4}-\varepsilon \right)n^2,$$
    which contradicts to  Eq (\ref{maxcut1}).
    Therefore it follows that
    $$\left (\frac{1}{2}-\sqrt{\varepsilon} \right)n\le |S|, |T|\le \left  (\frac{1}{2}+\sqrt{\varepsilon} \right)n.$$
    Hence the assertion holds.
        \end{proof}

\begin{lemma}\label{Lupper} Let $k\geq 2$.
Denote by
\[
L:= \left\{v: d(v) \leq \left(\frac{1}{2}-\frac{1}{4(k+1)}\right) n\right\}.
\]
Then
$$|L|\le 16k^2.$$
\end{lemma}
\begin{proof}
 Suppose that $|L| >16k^2$.
 Then let $L^{\prime}\subseteq L$ with $|L^{\prime}|=16k^2$.
 Then  it follows that
   \begin{eqnarray*}
e(G-L^{\prime}) &\ge& e(G)-\sum_{v\in L^{\prime}}d(v)\\
&\ge& \frac{n^2}{4}-2kn-16k^2 \left(\frac{1}{2}-\frac{1}{4(k+1)} \right)n\\
&>& \frac{(n-16k^2)^2}{4}+k^2.
     \end{eqnarray*}
 for $n$  a sufficiently large constant depending only on $k$, where the first inequality is by \eqref{maxsize}. Hence by Theorem~\ref{Erdos95}, $G-L^{\prime}$ contains $F_k$, which implies that $G$ contains $F_k$.
 So the assertion holds.
 \end{proof}

We will also need the following lemma which can be proved by induction or double counting.

\begin{lemma}\label{set}
Let $A_1, \cdots, A_p$ be $p$  finite sets. Then
\begin{equation}\label{set1}
|A_1\cap A_2\cap\cdots\cap A_p|\ge \sum_{i=1}^p|A_i|-(p-1)\left|\bigcup_{i=1}^pA_i\right|.
\end{equation}
\end{lemma}

For a vertex $v$, let $d_S(v) = |N(v) \cap S|$ and $d_T(v) = |N(v) \cap T|$, and let
\[
W: = \left\{ v\in S: d_S(v) \geq \delta n\right\} \cup \left\{v \in T: d_T(v) \geq \delta n\right\}
\]
be the set of vertices that have many neighbors which are not in the cut. Let $L$ be as in Lemma \ref{Lupper}, that is
\[
L= \left\{v: d(v) \leq \left(\frac{1}{2}-\frac{1}{4(k+1)}\right) n\right\}.
\]

Next we show that actually $W$ and $L$ are empty.

\begin{lemma}\label{W-Lenpty}
For the above $W$, we have
$$|W| <\frac{2\varepsilon}{\delta} n+\frac{2k^2}{\delta n},$$
and
$W \setminus L$ is empty.
\end{lemma}
\begin{proof}
 Since $e(S,T) \geq \left(\frac{1}{4} - \varepsilon \right)n^2$ and $e(G)\le \mathrm{ex} (n, F_k)\le \frac{n^2}{4}+k^2$ by Theorem \ref{Erdos95}.
 Hence
 \begin{equation}\label{WL-1}
 e(S)+e(T)=e(G)-e(S, T)\le \frac{n^2}{4}+k^2- \left(\frac{1}{4} - \varepsilon \right)n^2=\varepsilon n^2+k^2.
 \end{equation}
 On the other hand, if we
 let $W_1=W\cap S$ and $W_2=W\cap T$, then we deduce
 $$2e(S) =\sum_{u\in S}d_{S}(u) \ge  \sum_{u\in W_1}d_S(u)\ge |W_1|\delta n, \ \  2e(T) = \sum_{u\in T}d_{T}(u)\ge  \sum_{u\in W_2}d_T(u)\ge |W_2|\delta n.$$
  So
  \begin{equation}\label{WL-2}
  e(S)+e(T)\ge (|W_1|+|W_2|)\frac{\delta n}{2}=\frac{|W|\delta n}{2}.
  \end{equation}
  By \eqref{WL-1} and \eqref{WL-2}, we get
  $$\frac{|W|\delta n}{2}\le \varepsilon n^2+k^2,$$
  i.e.,
  \begin{equation}\label{W upper bound}|W|\le \frac{2(\varepsilon n^2+k^2)}{\delta n}.\end{equation}

Suppose that $W\setminus L \neq\emptyset$.  We now prove that this is impossible.

   Let $L_1=L\cap S$ and $L_2=L\cap T$. Without loss of generality, there exists a vertex $u\in W_1\setminus L_1.$ Since  $S$ and $T$ form a maximum cut, $d_T(u)\ge \frac{1}{2}d(u)$. On the other hand, $u\not\in L$ because $u\in W_1\setminus{L_1}$. Therefore $d(u)\ge \left(\frac{1}{2}-\frac{1}{4(k+1)} \right)n$. So
   $$d_T(u)\ge \frac{1}{2}d(u)\ge \left(\frac{1}{4}-\frac{1}{8(k+1)} \right)n.$$
  On the other hand,  $|L|\le 16k^2$.
  Hence, for fixed $\delta<\frac{1}{2k}$,  $\varepsilon<\frac{\delta^2}{2}$ and sufficiently large $n$, we have
  $$
  |S\setminus (W\cup L)|\ge \left(\frac{1}{2}- \sqrt{\varepsilon } \right)n-\delta n-\frac{2k^2}{\delta n}-16k^2\ge \left (\frac{1}{2}-\sqrt{\varepsilon } -\delta \right)n-18k^2\ge k.
  $$

   Suppose that $u$ is adjacent to   $k$ vertices  $u_1, \ldots, u_k$ in $S\setminus (W\cup L)$.  Since $u_i\not\in L$,  we have $d(u_i)\ge \left(\frac{1}{2}-\frac{1}{4(k+1)}\right )n$. On the other hand, $d_S(u_i)\le \delta n$ by $u_i\notin W$. So
  $
  d_T(u_i)=d(u_i)-d_S(u_i)\ge \left(\frac{1}{2}-\frac{1}{4(k+1)}-\delta \right)n
  $.
   By Lemma~\ref{set}, we have
   \begin{eqnarray*}
&& |N_T(u)\cap N_T(u_1)\cap \cdots \cap N_T(u_k)|\\
 &\ge&
    |N_T(u)|+|N_T(u_1)|+\ldots+|N_T(u_k)|-k |N_T(u)\cup N_T(u_1)\cup\cdots\cup N_T(u_k)|\\
 &\ge&  d_T(u)+ d_T(u_1)+\cdots +d_T(u_k) -k|T| \\
& \ge &\left (\frac{1}{4}-\frac{1}{8(k+1)}\right)n+\left(\frac{1}{2}-\frac{1}{4(k+1)}-\delta\right)n\cdot k-k \left(\frac{1}{2}+\sqrt{\varepsilon}\right)n\\
 &= &\left(\frac{1}{8(k+1)}-k\delta- k \sqrt{\varepsilon}\right)n>k
\end{eqnarray*}
for sufficiently large $n$. So there exist $k$  vertices $v_1, \ldots v_k$  in $T$ such that the subgraph induced by vertex set $\{u_1, \ldots, u_k, v_1, \ldots, v_k\}$ is complete bipartite.
It follows that $G$ contains $F_k$,
this is a contradiction.
Therefore $u$ is adjacent to at most $k-1$ vertices in  $S\setminus (W\cup L)$.
Hence, in  view of    $\varepsilon<\frac{\delta^2}{3}$, we have
   \begin{eqnarray*}
   d_S(u)&\le & |W|+|L|+k-1\\
   &<&  \frac{2\varepsilon}{\delta} n+\frac{2k^2}{\delta n}+16k^2+k-1\\
&<& \frac{2\delta}{3}n+\frac{2k^2}{\delta n}+17k^2\\
&<& \delta n
\end{eqnarray*}
for sufficiently large $n$.  This is a contradiction to the fact that $u\in W$.
Similarly, there  is no vertex  $u\in W_2\setminus L_2$,  Hence $W\setminus L=\emptyset.$
 \end{proof}

\begin{lemma}\label{Lempty}
$L$ is empty.
\end{lemma}
\begin{proof} We will prove the result from the following two claims.

  {\bf Claim 1}: There exist  independent sets
 $I_S\subseteq S$  and $I_T\subseteq T$ such that
  $$|I_S|\ge  |S|-18k^2, \quad
 \mbox {and} \quad
  |I_T|\ge |T|-18k^2.
 $$

Indeed, let  $u_1, \ldots, u_{2k} \in S\setminus L$. Then $u_i\notin L$ which implies
$$
d(u_i)\ge \left(\frac{1}{2}-\frac{1}{4(k+1)} \right)n.
$$
By Lemma~\ref{W-Lenpty}, $d_S(u_i)\le \delta n$. Hence
$$d_T(u_i)=d(u_i)-d_S(u_i)\ge \left(\frac{1}{2}-\frac{1}{4(k+1)}-\delta\right)n.$$

 Furthermore by Lemma~\ref{set}, we  have
    \begin{eqnarray*}
\left |\bigcap_{i=1}^{2k}N_T(u_i) \right| &\ge &  \sum_{i=1}^{2k}|N_T(u_i)|-(2k-1) \left|\bigcup_{i=1}^{2k}N_T(u_i) \right| \\
& \ge& \left (\frac{1}{2}-\frac{1}{4(k+1)}-\delta \right)n\cdot 2k-(2k-1)(\frac{1}{2}+\sqrt{\varepsilon})n \\
 &=&\left(\frac{1}{2(k+1)}-2k\delta-(2k-1)\sqrt{\varepsilon} \right)n\\
 &>&k
\end{eqnarray*}
for sufficiently  large $n$.
Hence there exist $k$ vertices $v_1, \ldots, v_k$ such that the induced subgraph by two partitions $\{u_1, \ldots, u_{2k}\}$ and $\{v_1, \ldots, v_k\}$ is a complete bipartite graph.  So $G[S\setminus L]$ are both $K_{1, k}$ and $M_k$-free, otherwise $G$ contains $F_k$, i.e., $uu_1v_1, \ldots uu_kv_k$ for $d(u)\ge k$, or $v_1u_1u_2, \ldots, v_1u_{2k-1}u_{2k}$
 for $\{u_1u_2, \ldots u_{2k-1}u_{2k}\}$  being a matching of size $k$.
Hence the maximum degree and the maximum matching number of $G[S\setminus L]$  are at most  $k-1$, respectively. By Theorem \ref{Chvatal76},
$$e(G[S\setminus L])\le f(k-1, k-1).$$
 The same argument gives
 $$e(G[T\setminus L])\le f(k-1, k-1).$$
 Since $G[S \setminus L]$ has at most $f(k-1, k-1)$ edges, then  the subgraph obtained from $G[S \setminus L]$ by deleting one vertex of each edge in $G[S\setminus L]$ contains no edges, which is an independent set of $G[S\setminus L]$. So there exists an independent set
 $I_S\subseteq S$ such that
  $$
  |I_S|\ge |S\setminus L|-f(k-1, k-1)\ge |S|-k(k-\frac32)-16k^2\ge |S|-18k^2.
  $$
 The same argument gives that there  is an independent set $I_T\subseteq T$ with
 $$ |I_T|\ge |T|-18k^2.
 $$
 So Claim 1 holds.

 Recall that $z$ is a vertex with maximum eigenvector entry. Since  $x_z=1$, and
 $$d(z)\ge \sum_{w\sim z}x_w=\lambda_1 x_z=\lambda_1\ge   \frac{n}{2},$$
 which implies $d(z)\ge \frac{n}{2}$.
 Hence $z\notin L$.

  Without loss of generality, we may assume that $z\in S$.
Since the maximum degree in the induced subgraph $G[S\setminus L]$  is at most $k-1$ (containing no $K_{1,k}$), from Lemma~\ref{Lupper}, we have $|L|\le 16k^2$ and
 $$
 d_S(z)=d_{S\cap L}(z)+d_{S\setminus L}(z)\le k-1+16k^2\le 17k^2.
 $$
 Therefore,  by Claim 1, we have
 \begin{eqnarray*}
 \lambda_1&=&\lambda_1x_z=\sum_{v\sim z}x_v\\
 &=& \sum_{v\sim z, v\in S}x_v+\sum_{v\sim z , v\in T}x_v\\
 &=& \sum_{v\sim z, v\in S}x_v+  \sum_{v\sim z, v\in I_T}x_v+\sum_{v\sim z , v\in T\setminus I_T}x_v\\
 &\le&  d_S(z)+\sum_{ v\in I_T}x_v+\sum_{ v\in T\setminus I_T}1\\
 &\le & 17k^2+\sum_{ v\in I_T}x_v+|T|-|I_T|\\
 &\le & \sum_{ v\in I_T}x_v  +17k^2+18k^2\\
&\le & \sum_{ v\in I_T}x_v  +35k^2.
\end{eqnarray*}
 So \begin{equation}\label{Lempty2}
 \sum_{ v\in I_T}x_v\ge \lambda_1-35k^2.
 \end{equation}

  {\bf Claim 2}:  $L=\emptyset$.

By way of contradiction, assume that there is a vertex  $v\in L$, i.e., $d(v)\le (\frac{1}{2}-\frac{1}{4(k+1)})n$.
Consider the graph $G^+$ with vertex set $V(G)$ and edge set $E(G^+) = E(G \setminus \{v\}) \cup \{vw: w\in I_T\}$. Note that adding a vertex incident with vertices in $I_T$ does not create any triangles, and so $G^+$ is $F_k$-free. By (\ref{Lempty2}),  we have that
\begin{align*}
\lambda_1(G^+) - \lambda_1(G) &\geq \frac{\mathbf{x}^T\left(A(G^+) - A(G)\right) \mathbf{x}}{\mathbf{x^T}\mathbf{x}} = \frac{2\mathbf{x}_v}{\mathbf{x}^T\mathbf{x}}\left( \sum_{w\in I_T} \mathbf{x}_w - \sum_{uv\in E(G)} \mathbf{x}_u\right) \\
& \geq \frac{2 \mathbf{x}_v}{\mathbf{x}^T\mathbf{x}} \left( \lambda_1 - 35k^2 - d_G(v)\right)\\
&\geq \frac{2 \mathbf{x}_v}{\mathbf{x}^T\mathbf{x}} \left( \lambda_1 - 35k^2 - \left (\frac{1}{2}-\frac{1}{4(k+1)} \right)n\right) \\
&\geq \frac{2 \mathbf{x}_v}{\mathbf{x}^T\mathbf{x}} \left(\frac{n}{2} -35k^2 -\left (\frac{1}{2}-\frac{1}{4(k+1)} \right)n\right) \\
&=\frac{2 \mathbf{x}_v}{\mathbf{x}^T\mathbf{x}} \left( \frac{n}{4(k+1)}-35k^2\right)> 0,
\end{align*}
where the last step uses $n$ large enough and that if $v\in L$, then $d_G(v) \leq \left(\frac{1}{2} - \frac{1}{4k+4}\right)n$. This contradicts $G$ has the largest spectral radius over all $F_k$-free graphs and so $L$ must be empty.
\end{proof}

Next we may refine the structure of $G$.

\begin{lemma}\label{STlambdarefine}
For $n$ and $k$ as before, we have
\begin{equation}\label{STrefine1}
\frac{n}{2}-4k
\le |S|,  |T|\le \frac{n}{2}+4k,
\end{equation}
\begin{equation}\label{STrefine2}
e(G)\ge \frac{n^2}{4}-12k^2,
\end{equation}
and
\begin{equation}
\label{lambdarefine}
\frac{n}{2}-14k^2\le \delta(G)\le \lambda_1\le \Delta (G)\le \frac{n}{2}+5k.
\end{equation}
\end{lemma}
\begin{proof}
 Since $L$ and $W$ are empty, we have that both $G[S]$ and $G[T]$ are $K_{1,k}$- and $M_k$-free, and so we have $e(S) + e(T) \leq 2f(k-1,k-1) < 2k^2$.
   This means that the number of triangles in $G$ is bounded above by $2k^2n$ since any triangle contains an  edge of  $E(S)\cup E(T)$.
       By Corollary \ref{edgelower2},
       we have $$e(G) \geq \lambda_1^2-\frac{6t}{n}\ge \frac{n^2}{4}  -\frac{12k^2n}{n  }=\frac{n^2}{4}-12 k^2.$$

       Suppose that $|S|\le \frac{n}{2}-4k,$ then $|T|=n-|S|\ge \frac{n}{2}+4k$.
       Hence
       $$e(G)=e(S)+e(T)+e(S,T)\le 2k^2+|S||T|\le 2k^2+ \left (\frac{n}{2}-4k \right) \left(\frac{n}{2}+4k \right)= \frac{n^2}{4}-14k^2,$$
        which contradicts to $e(G)\geq  \frac{n^2}{4}-12k^2$.

        So we have
        $$
        \frac{n}{2}-4k
\le |S|, |T|\le \frac{n}{2}+4k.
$$

Moreover, the maximum degree of $G[S]$  is at most $k-1$, since it contains no $K_{1, k}$, otherwise,  by the proof of Lemma \ref{W-Lenpty}, i.e.,
for $u, u_1, \ldots, u_k\in S$ with $uu_1, \ldots, uu_k$ being edges, we obtain
$$| N_T(u)\cap N_T(u_1)\cap\cdots \cap N_T(u_k)|\ge k.$$
So there exist $v_1, \ldots, v_k \in T$ such that  the  subgraph induced by  $u, u_1, \ldots, u_k, v_1, \ldots, v_k$  is a complete bipartite graph.
 So we have one  $F_k$, that is, $uu_1v_1, \ldots, uu_kv_k$.
This implies that
 $$\Delta(G)\le \frac{n}{2}+4k+k-1\le \frac{n}{2}+5k.$$
  So $$\lambda_1\le \Delta(G)\le \frac{n}{2}+5k.$$
Furthermore, we claim that the minimum degree of $G$ is at least $\frac{n}{2} - 14k^2$.  Otherwise,
 removing a vertex  $v$ of minimum  degree $d(v)$, we have
\begin{eqnarray*}
 e(G-v)&=&e(G)-d(v)\\
 &\ge & \frac{n^2}{4}-12k^2- \left( \frac{n}{2}-14k^2 \right)\\
 &= & \frac{n^2}{4}-\frac{n}{2}+2k^2 \\
 &=& \frac{(n-1)^2}{4}+k^2+k^2-\frac{1}{4}\\
 &> & \frac{(n-1)^2}{4}+k^2,
\end{eqnarray*}
which implies $G-v$ contains $F_k$ by Theorem  \ref{Erdos95}.
\end{proof}

\begin{lemma}\label{second lower bound evector2}
For all $u\in V(G)$,  we have that $\mathbf{x}_u\geq 1-\frac{116 k^2}{n}$.
\end{lemma}

\begin{proof} Without loss of generality, we may  assume that $z\in S$.  We consider the following two cases.

{\bf Case 1:}    $u\in S$.  Then $d_S(u)\le k^2 $ as  $e(G[S])\le k^2$. Hence we obtain
\begin{eqnarray*}
|N_T(u)|&=&d_T(u)= d(u)-d_S(u)\ge \delta(G)-d_S(u)\ge \frac{n}{2}-14k^2-k^2\\
&=&\frac{n}{2}-15k^2.\\ [2mm]
|N_T(u)\cap N_T(z)| &=& |N_T(u)|+|N_T(z)|-|N_T(u)\cup N_T(z)|\ge 2\delta(G)-|T|\\
&\ge& 2 \left(\frac{n}{2}-14 k^2 \right)- \left(\frac{n}{2}+4k \right)\ge \frac{n}{2}-32 k^2.\\[2mm]
\lambda_1x_u-\lambda_1x_z &=& \sum_{v\sim u, v\in T, v\sim z}x_v+\sum_{v\sim u, v\in T, v\not\sim z}x_v+\sum_{v\sim u, v\in S}x_v\\
& &- \sum_{v\sim z, v\in T, v\sim u}x_v-\sum_{v\sim z, v\in T, v\not\sim u}x_v-\sum_{v\sim z, v\in S}x_v\\
&\ge & -\sum_{v\sim z, v\in T, v\not\sim u}x_v-\sum_{v\sim z, v\in S}x_v\\
&\ge & -\sum_{v\sim z, v\in T,  v\not\sim u}1-\sum_{v\sim z, v\in S}1\\
&\ge & -(d_T(z)-|N_T(u)\cap N_T(z)|)- d_S(z)\\
&\ge & - \left( \left(\frac{n}{2}+5k \right)-  \left(\frac{n}{2}-32k^2\right) \right)- k\\
&\ge & -38 k^2.
\end{eqnarray*}
Therefore,  for any $u\in S$, we have
\begin{equation}\label{verct1}
x_u\ge 1-\frac{38k^2}{\lambda_1}> 1-\frac{38k^2}{ \frac{n}{2}}=1-\frac{76k^2}{n}.
\end{equation}

{\bf Case 2:}   $u\in T$. By (\ref{verct1}),
 $$\lambda_1x_u=\sum_{v\sim u}x_v\ge \sum_{v\sim u, v\in S}x_v\ge \left(1-\frac{76k^2}{n} \right)d_S(u).$$
 Since
 $$
 \frac{n}{2}-14k^2\leq \delta(G)\le d(u)=d_S(u)+d_T(u),
 $$
  and $d_T(u)\le k$  as  the maximum degree  in $G[T]$ is at most $k-1$,
  we have $d_S(u)\ge \frac{n}{2}-14k^2-k\ge  \frac{n}{2}-15k^2$.
 Hence
 \begin{eqnarray*}
 x_u&\ge & \frac{(1-\frac{76k^2}{n})d_S(u)}{\lambda_1} \ge \frac{(1-\frac{76 k^2}{n})(\frac{n}{2}-15k^2)}{\frac{n}{2}+5k} \\
& =& \frac{\frac{n}{2}-53k^2+\frac{1140k^4}{n}}{\frac{n}{2}+5k} \\
 &=&1-\frac{53k^2+5k-\frac{1140 k^4}{n}}{\frac{n}{2}+5k}\\
 &>& 1-\frac{58 k^2}{\frac{n}{2}}=1-\frac{116k^2}{n}.
 \end{eqnarray*}

 From the above two cases, the result follows.
\end{proof}

Using this refined bound on the eigenvector entries, we may show that the partition $V=S\cup T$ is balanced.

\begin{lemma}\label{cut balanced}
The sets $S$ and $T$ have sizes as close as possible. That is
\[
\big||S| - |T|\big| \leq 1.
\]
\end{lemma}

\begin{proof}
Without loss of generality, we may assume that $|T| \geq |S|$. Denote
\begin{align*}
S' &:= \{v\in S: N(v) \subseteq T\},\\
T' &:= \{v\in T: N(v) \subseteq S\}.
\end{align*}
Since $e(G[S]) \le k^2$, there exist at most $2k^2$ vertices in $S$   having a neighbor in $S$. Hence
$$
|S'|\ge |S|-2k^2.
$$
Similarly,
$$|T'|\ge |T|-2k^2.
$$
 Let $C\subseteq T'$ be a set having $|T|-|S|$ vertices. Because,  from (\ref{STrefine1}),  $|T|-|S|\le 8k$ and $|T'|\ge |T|-2k^2\geq  \frac{n}{2} -4k -2k^2>8k$.
 Then $G\setminus C$ is a graph on $2|S|$ vertices such that
$$e(G)-e(C,S)=e(G\setminus C)\le \mathrm{ex} (2|S|, F_k)\le\frac{(2|S|)^2}{4}+f(k-1, k-1).$$
Hence
$$e(G)\le |S|^2+|C||S|+f(k-1, k-1)=|S||T|+f(k-1, k-1).$$

Let $B=K_{|S|, |T|}$ be the complete bipartite graph with partite sets $S$ and $T$, and let $G_1 = G[S] \cup G[T]$ and $G_2$ be the graph with edges $E(B) \setminus E(G)$. Note that $e(G) = e(B) + e(G_1) - e(G_2)$ and so
$e(G_1) - e(G_2) =e(G)-e(B)
\leq f(k-1,k-1)$.
Note also that $e(G_1) \leq 2k^2$, then by Lemma \ref{second lower bound evector2} we have, for sufficiently large $n$,
$$
\mathbf{x}^T\mathbf{x} \geq n \left(1-\frac{116k^2}{n} \right)^2 > n \left(1-\frac{232k^2}{n} \right)=n-232k^2,
$$
 and that $\lambda_1(B) = \sqrt{|S||T|}$.  By Lemma~\ref{STlambdarefine},  we obtain
 $$e(S, T)=e(G)-e(G_1)\ge \frac{n^2}{4}-12k^2-2k^2=\frac{n^2}{4}-14k^2,$$
 which implies that
 $$
 e(G_2)=e(B)-e(S,T)\le |S||T|- \left(\frac{n^2}{4}-14k^2 \right)\le 14k^2.
 $$
 So  we have
\begin{align*}
&\frac{2}{n} \left \lfloor\frac{n}{2} \right \rfloor \left\lceil\frac{n}{2} \right\rceil + \frac{ 2 f(k-1,k-1)}{n}\leq \\
  \lambda_1
 &= \frac{\mathbf{x}^T (A(B) + A(G_1) - A(G_2))\mathbf{x}}{\mathbf{x}^T\mathbf{x}} \\
  &=  \frac{\mathbf{x}^T A(B)\mathbf{x}}{\mathbf{x}^T\mathbf{x}}+ \frac{\mathbf{x}^T A(G_1)\mathbf{x}}{\mathbf{x}^T\mathbf{x}}-\frac{\mathbf{x}^T A(G_2)\mathbf{x}}{\mathbf{x}^T\mathbf{x}}
  \\
  & \leq \lambda_1(B) + \frac{2e(G_1)}{\mathbf{x}^T\mathbf{x}}  -\frac{2e(G_2)(1-\frac{232 k^2}{n})}{\mathbf{x}^T\mathbf{x}}\\
   & \leq \lambda_1(B) + \frac{2(e(G_1) -e(G_2))}{\mathbf{x}^T\mathbf{x}}+ \frac{2e(G_2)\frac{232k^2}{n}}{\mathbf{x}^T\mathbf{x}}\\
  & \leq \sqrt{|S||T|} + \frac{2 f(k-1, k-1)}{\mathbf{x}^T\mathbf{x}} + \frac{{2\cdot 14k^2}\frac{232k^2}{n}} {\mathbf{x}^T\mathbf{x}}\\
   & \leq \sqrt{|S||T|} + \frac{2 f(k-1, k-1)}{n} +2f(k-1, k-1)(\frac{1}{\mathbf{x}^T\mathbf{x}}-\frac{1}{n})+  \frac{{2\cdot 14k^2}\frac{232k^2}{n}} {n(1-\frac{232 k^2}{n})}\\
    & \leq \sqrt{|S||T|} + \frac{2 f(k-1, k-1)}{n} +2k^2(\frac{1}{n(1-\frac{232k^2}{n})}-\frac{1}{n})+  \frac{{2\cdot 14k^2}\frac{232k^2}{n}} {n(1-\frac{232k^2}{n})}\\
     & = \sqrt{|S||T|} + \frac{2 f(k-1, k-1)}{n}+\frac{464k^4}{n(n-232k^2)}+\frac{6496k^4}{n(n-232 k^2)}\\
     & = \sqrt{|S||T|} + \frac{2 f(k-1, k-1)}{n}+\frac{6960 k^4}{n(n-232 k^2)}.\\
 \end{align*}
 Then
 \begin{equation}\label{last}
 \frac{2}{n} \left \lfloor\frac{n}{2} \right \rfloor  \left\lceil\frac{n}{2} \right \rceil    -\sqrt{|S||T|}\le \frac{6960 k^4}{n(n-232 k^2)}.
 \end{equation}
 Suppose that $|T|\ge |S|$ and $|T|\ge |S|+2$.
 We consider two cases.

 {\bf Case 1:} $n$ is even. Since $|S|+|T|=n$, we  have
 \begin{eqnarray*}
  \frac{2}{n} \left \lfloor\frac{n}{2} \right \rfloor  \left\lceil\frac{n}{2} \right \rceil  -\sqrt{|S||T|} & \ge&  \frac{n}{2}-\sqrt{ \left (\frac{n}{2}-1 \right) \left(\frac{n}{2}+1 \right)}\\
&=&\frac{n}{2}-\sqrt{\frac{n^2}{4}-1}=\frac{1}{\frac{n}{2}+\sqrt{\frac{n^2}{4}-1}}>\frac{1}{n}.
  \end{eqnarray*}
 So by (\ref{last}), we have
 $$\frac{1}{n}< \frac{2}{n}\left \lfloor\frac{n}{2} \right \rfloor  \left\lceil\frac{n}{2} \right \rceil -\sqrt{|S||T|}\le  \frac{6960 k^4}{n(n-232 k^2)} \le \frac{7000 k^4}{n^2}.
 $$
 It is a contradiction for sufficiently  large $n$.

 {\bf Case 2:} $n$ is odd. Since $|S|+|T|=n$, we  have
   \begin{eqnarray*}
    \frac{2}{n}\left \lfloor\frac{n}{2} \right \rfloor  \left\lceil\frac{n}{2} \right \rceil -\sqrt{|S||T|}
    &\ge & \frac{n^2-1}{2n}-\sqrt{\left (\frac{n-3}{2} \right) \left(\frac{n+3}{2} \right)}\\
&=& \frac{n-\frac{1}{n}}{2}-\frac{\sqrt{n^2-9}}{2}=\frac{(n-\frac{1}{n})^2-(n^2-9)}{2(n-\frac{1}{n}+\sqrt{n^2-9})}\\
&=& \frac{7+\frac{1}{n^2}}{2(n-\frac{1}{n}+\sqrt{n^2-9})}\ge \frac{1}{n}.
  \end{eqnarray*}
 So by (\ref{last}), we have
 $$\frac{1}{n}< \frac{2}{n} \left \lfloor\frac{n}{2} \right \rfloor  \left\lceil\frac{n}{2} \right \rceil-\sqrt{|S||T|}\le\frac{6960 k^4}{n(n-232 k^2)} \le \frac{7000 k^4}{n^2}.
 $$
It is a contradiction for sufficiently large $n$.
Therefore
for $n$ large enough we must have that $||S| - |T|| \leq 1$.
\end{proof}

Finally, we show that $e(G) = \mathrm{ex}(n, F_k)$.

{\bf  Proof of Theorem \ref{MainTH}}.
 By way of contradiction, we assume that $e(G) \leq \mathrm{ex}(n, F_k) - 1$. Let $H$ be an $F_k$-free graph  with $\mathrm{ex}(n, F_k)$ edges on the same vertex set as $G$, where $S$ and $T$ induce a complete bipartite graph in $H$ (this is possible because every graph in $\mathrm{Ex}(n, F_k)$ has a maximum cut of size $\lfloor n^2/4\rfloor$).  Let $E_+$ and $E_-$ be sets of edges such that $E(G) \cup E_+ \setminus E_- = E(H)$, and choose $E_+$ and $E_-$ to be as small as possible, i.e., $E_+=E(H)\setminus E(G)$ and $E_-=E(G)\setminus E(H)$. Since $|E(G)\cap E(H)|+|E_-|=e(G)< e(H)=|E(G)\cap E(H)+|E_+|$ which implies that
    $|E_+| \geq |E_-| + 1$. Furthermore, we have that $|E_-| \leq e(S) + e(T) < 2k^2$. Now, by the Rayleigh quotient \cite{horn}, we have that
\begin{align*}
\lambda_1(H) &\geq \frac{\mathbf{x}^T A(H) \mathbf{x}}{\mathbf{x}^T\mathbf{x}} = \lambda_1(G) +\frac{2}{\mathbf{x}^T\mathbf{x}} \sum_{ij\in E_+} \mathbf{x}_i\mathbf{x}_j - \frac{2}{\mathbf{x}^T\mathbf{x}} \sum_{ij\in E_-} \mathbf{x}_i\mathbf{x}_j \\
& \geq \lambda_1(G) + \frac{2}{\mathbf{x}^T\mathbf{x}} \left( |E_+|\left(1 - \frac{116k^2}{n}\right)^2 - |E_-|\right)\\
& \geq \lambda_1(G) + \frac{2}{\mathbf{x}^T\mathbf{x}} \left( |E_+|  - |E_-|-\frac{232k^2}{n}  |E_+| +\frac{(116k^2)^2}{n^2}  |E_+| \right)\\
& \geq \lambda_1(G) + \frac{2}{\mathbf{x}^T\mathbf{x}} \left( 1-\frac{232k^2}{n}  |E_+| +\frac{(116k^2)^2}{n^2}  |E_+| \right)\\
&>\lambda_1(G),
\end{align*}
for sufficiently  large $n$, where we are using that $|E_-| < 2k^2$ and $|E_+| \geq |E_-| + 1$. Therefore we have that for $n$ large enough, $\lambda_1(H) > \lambda_1(G)$, a contradiction.
Hence $e(G)=e(H)$.

From the above discussion, we complete the proof of Theorem \ref{MainTH}. $\blacksquare$

\end{document}